\documentclass[a4paper,12pt]{article}

\usepackage{amssymb}
\usepackage{latexsym}
\usepackage{amsmath}
\usepackage{amsthm}
\usepackage{amsfonts}
\usepackage{color}
\usepackage{pdfsync}
\usepackage[usenames,dvipsnames]{pstricks}
\usepackage{epsfig}
\usepackage{pst-grad} 
\usepackage{pst-plot} 

\numberwithin{equation}{section}

\usepackage{amssymb}

\newcommand{\beq}{\begin{equation} }
\newcommand{\eqq}{\end{equation} }
\newcommand{\cuad}{{\sqcap\kern-.68em\sqcup}}

\newtheorem{remark}{Remark}[section]
\newcommand{\bremark}{\begin{remark} \em}
\newcommand{\eremark}{\end{remark} }

\def\beeq{\begin{equation}}
\def\eeq{\end{equation}}
\newcommand{\begeqaet}{\begin{eqnarray*}}
\newcommand{\eneqaet}{\end{eqnarray*}}

\let\Section=\section
\def\section{\setcounter{equation}{0}\Section}
\newtheorem{Lem}{Lemma}[section]
\newtheorem{Thm}{Theorem}[section]

\begin{document}
\begin{center}{\bf\Large Existence and concentration of solution for a class of fractional Hamiltonian systems with subquadratic potential}\medskip

\bigskip

\bigskip

{C\'esar E. Torres Ledesma}

 Departamento de Matem\'aticas, \\
Universidad Nacional de Trujillo\\
Av. Juan Pablo Segundo s/n Trujillo, Per\'u.\\
 {\sl  (ctl\_576@yahoo.es, ctorres@dim.uchile.cl)}


\medskip

\medskip
\medskip
\medskip
\medskip

\end{center}

\centerline{\bf Abstract}
This article study the fractional Hamiltonian systems
\begin{eqnarray}\label{00}
{_{t}}D_{\infty}^{\alpha}({_{-\infty}}D_{t}^{\alpha}u) + \lambda L(t)u = \nabla W(t, u), \;\;t\in \mathbb{R},
\end{eqnarray}
where $\alpha \in (1/2, 1)$, $\lambda >0$ is a parameter, $L\in C(\mathbb{R}, \mathbb{R}^{n\times n})$ and $W \in C^{1}(\mathbb{R} \times \mathbb{R}^n, \mathbb{R})$. Unlike most other papers on this problem, we require that $L(t)$ is a positive semi-definite symmetric matrix for all $t\in \mathbb{R}$, that is, $L(t) \equiv 0$ is allowed to occur in some finite interval $\mathbb{I}$ of $\mathbb{R}$. Under some mild assumptions on $W$, we establish the existence of nontrivial weak solution, which vanish on $\mathbb{R} \setminus \mathbb{I}$ as $\lambda \to \infty,$ and converge to $\tilde{u}$ in $H^{\alpha}(\mathbb{R})$; here $\tilde{u} \in E_{0}^{\alpha}$ is nontrivial weak solution of the Dirichlet BVP for fractional Hamiltonian systems on the finite interval $\mathbb{I}$. 

\medskip

\medskip

\date{}

\setcounter{equation}{0}
\section{ Introduction}

In this paper we investigate the solvability  of the following non homogeneous fractional Hamiltonian system 
\begin{eqnarray}\label{01}
{_{t}}D_{\infty}^{\alpha}({_{-\infty}}D_{t}^{\alpha}u) + \lambda L(t)u = \nabla W(t, u), \;\;t\in \mathbb{R},
\end{eqnarray}
where $\alpha \in (1/2,1)$, $W\in C(\mathbb{R} \times \mathbb{R}^{n}, \mathbb{R})$, the parameter $\lambda >0$, ${_{-\infty}}D_{t}^{\beta}$ and ${_{t}}D_{\infty}^{\beta}$ denote left and right Liouville-Weyl fractional derivative of order $\alpha$ respectively and are defined by
$$
{_{-\infty}}D_{t}^{\beta}u = \frac{d}{dt}{_{-\infty}}I_{t}^{\alpha}u,\;\; {_{t}}D_{\infty}^{1-\alpha}u = -\frac{d}{dt}{_{t}}I_{\infty}^{1-\alpha}.
$$
and the matrix $L$ satisfies the following conditions:
\begin{enumerate}
\item[($L_1$)] $L(t)\in C(\mathbb{R}, \mathbb{R}^{n\times n})$ is a symmetric matrix for all $t\in \mathbb{R}$;  there exists a nonnegative continuous function $l:\mathbb{R} \to \mathbb{R}$ and a constant $k>0$ such that 
$$
(L(t)u(t), u(t)) \geq l(t)|u(t)|^2,
$$
and the set $\{l<k\} = \{t\in \mathbb{R}:\;l(t) < k\}$ is nonempty with $C_{\alpha}^2|\{l<k\}| < 1$, where $|.|$ is the Lebesgue measure and $C_{\alpha}$ is the Sobolev constant (see section \S 2). 
\item[($L_2$)] $\mathbb{J} = int(l^{-1}(0))$ is a nonempty finite interval and $\overline{\mathbb{J}} = l^{-1}(0)$.
\item[($L_3$)] There exists an open interval $\mathbb{I} \subset \mathbb{J}$ such that $L(t) \equiv 0$ for all $t\in \overline{\mathbb{I}}$.
\end{enumerate}

Fractional differential equations appear naturally in a number of fields such as physics, chemistry, biology, economics, control theory, signal and image processing, blood flow phenomena, etc. During last decades, the theory of fractional differential equations is an area intensively developed, due mainly to the fact that fractional derivatives provide an excellent tool for the description of memory and hereditary properties of various materials and processes (see \cite{RHe, RH, AKHSJT, IP, BWMBPG} and the references therein). 

Physical models containing left and right fractional differential operators have recently renewed attention from scientists which is mainly due to applications as models for physical phenomena exhibiting anomalous diffusion (see \cite{DBRSMM, DBSWMM1, DBSWMM2, JFJR, JLTB, AMCT, ES, CT, CT1, CT2}). A strong motivation for investigating the fractional differential equation (\ref{01}) comes from symmetry fractional advection-dispersion equation. A fractional advection-dispersion equation is a generalization of the classical ADE in which the second-order derivative is replaced with a fractional-order derivative. In contrast to the classical ADE, the fractional ADE has solutions that resemble the highly skewed and heavy-tailed breakthrough curves observed in field and laboratory studies \cite{DBRSMM, DBSWMM1}, in particular in contaminant transport of ground-water flow \cite{DBSWMM2}. In \cite{DBSWMM2}, the authors stated that solutes moving through a highly heterogeneous aquifer violations violates the basic assumptions of local second-order theories because of large deviations from the stochastic process of Brownian motion. Moreover, models involving a fractional differential oscillator equation, which contains a composition of left and right fractional derivatives, are proposed for the description of the processes of emptying the silo \cite{JLTB} and the heat flow through a bulkhead filled with granular material \cite{ES}, respectively. Their studies show that the proposed models based on fractional calculus are efficient and describe well the processes.


Very recently, in \cite{CT} the author considered (\ref{01}), 
where $L\in C(\mathbb{R}, \mathbb{R}^{n^2})$ is a symmetric matrix valued function for all $t\in \mathbb{R}$, $W\in C^{1}(\mathbb{R}\times \mathbb{R}^{n}, \mathbb{R})$ and $\nabla W(t, u(t))$ is the gradient of $W$ at $u$. Assuming that $L$ and $W$ satisfy the following hypotheses:
\begin{itemize}
\item[$(L)$] $L(t)$ is positive definite symmetric matrix for all $t\in \mathbb{R}$, and there exists an $l\in C(\mathbb{R}, (0,\infty))$ such that $l(t) \to +\infty$ as $t \to \infty$ and
    \begin{equation}\label{Eq03}
    (L(t)x,x) \geq l(t)|x|^{2},\;\;\mbox{for all}\;t\in \mathbb{R}\;\;\mbox{and}\;\;x\in \mathbb{R}^{n}.
    \end{equation}
\item[$(W_{1})$] $W\in C^{1}(\mathbb{R} \times \mathbb{R}^{n}, \mathbb{R})$, and there is a constant $\mu >2$ such that
$$
0< \mu W(t,x) \leq (x, \nabla W(t,x)),\;\;\mbox{for all}\;t\in \mathbb{R}\;\;\mbox{and}\;x\in\mathbb{R}^{n}\setminus \{0\}.
$$
\item[$(W_{2})$] $|\nabla W(t,x)| = o(|x|)$ as $x\to 0$ uniformly with respect to $t\in \mathbb{R}$.
\item[$(W_{3})$] There exists $\overline{W} \in C(\mathbb{R}^{n}, \mathbb{R})$ such that
$$
|W(t,x)| + |\nabla W(t,x)| \leq |\overline{W(x)}|\;\;\mbox{for every}\;x\in \mathbb{R}^{n}\;\mbox{and}\;t\in \mathbb{R}.
$$
\end{itemize}
The paper \cite{CT} 
 showed that (\ref{01}) has at least one nontrivial solution via Mountain pass theorem.

In particular, if $\alpha = 1$, (\ref{01}) reduces to the standard second order differential equation
\begin{equation}\label{HEq01}
u'' - L(t)u + \nabla W(t,u)=0,
\end{equation}
where $W: \mathbb{R} \times \mathbb{R}^{n} \to \mathbb{R}$ is a given function and $\nabla W(t,u)$ is the gradient of $W$ at $u$. The existence of homoclinic solution is one of the most important problems in the history of that kind of equations, and has been studied intensively by many mathematicians. Assuming that $L(t)$ and $W(t,u)$ are independent of $t$, or $T$-periodic in $t$, many authors have studied the existence of homoclinic solutions for (\ref{HEq01}) via critical point theory and variational methods. In this case, the existence of homoclinic solution can be obtained by going to the limit of periodic solutions of approximating problems.

If $L(t)$ and $W(t,u)$ are neither autonomous nor periodic in $t$, this problem is quite different from the ones just described, because the lack of compactness of the Sobolev embedding. In \cite{PRKT} the authors considered (\ref{HEq01}) without periodicity assumptions on $L$ and $W$ and showed that (\ref{HEq01}) possesses one homoclinic solution by using a variant of the mountain pass theorem without the Palais-Smale condition. In \cite{WOMW}, under the same assumptions of \cite{PRKT}, the authors, by employing a new compact embedding theorem, obtained the existence of homoclinic solution of (\ref{HEq01}).

Motivated by the previously mentioned results, using the genus properties of critical point theory, in \cite{ZZRY}, the authors generalized the result of \cite{CT} and established some new criterion to guarantee the existence of infinitely many solutions of (\ref{01}) for the case that $W(t,u)$ is subquadratic as $|u| \to +\infty$. 

As is well-known, the condition $(L)$ is the so-called coercive condition and is a little demanding. In fact, for a simple choice like $L(t) = sId_{n}$, the condition ($L$) is not satisfied, where $s>0$ and $Id_{n}$ is the $n\times n$ identity matrix. Considering this trouble, in \cite{ZZRY1}  the recent results in \cite{ZZRY} are generalized and significantly improved. More precisely  in \cite{ZZRY1} the authors considered the case that L(t) is bounded in the sense that
\begin{itemize}
\item[$(L)'$] There are constants $0 < \tau_{1} < \tau_{2} <+\infty$ such that
$$
\tau_{1}|u|^2 \leq (L(t)u,u) \leq \tau_{2}|u|^2\;\;\mbox{for all} \;\;(t,u)\in \mathbb{R}\times \mathbb{R}^n.
$$
\end{itemize}
Again, using the genus properties of critical point theory, the authors proved that (\ref{01}) has infinitely many nontrivial solutions. 

Very recently, using the fountain theorem of critical point theory, in \cite{JXDOKZ}, the authors established the existence of infinitely many solutions of (\ref{01}) for the case that $W(t,u)$ is subquadratic as $|u| \to 0$ and superqudratic as $|u| \to \infty$.     


Motivated by the above articles, we continue to consider problem (\ref{01}) with the positive semi-definite matrix $L$ and study the existence of nontrivial weak solutions when $W$ is sub-quadratic. Furthermore, more importantly, we shall explore the phenomenon of concentrations of weak solution as $\lambda \to \infty$, which seems to be rarely concerned in the previous studies of solutions for fractional Hamiltonian systems. To reduce our statements, we make the following assumptions on $W$:
\begin{enumerate}
\item[($W_1$)] $W\in C^1(\mathbb{R} \times \mathbb{R}^{n}, \mathbb{R})$ and there exist a constant $p\in (1,2)$ and function $\xi(t) \in L^{\frac{2}{2-p}}(\mathbb{R}, \mathbb{R}^+)$ such that
\begin{equation}\label{05}
|\nabla W(t, u)|\leq \xi(t)|u|^{p-1}, \;\;\mbox{for all}\;\;(t,u)\in \mathbb{R} \times \mathbb{R}^n.
\end{equation}
\item[($W_2$)] There exist three constants $\eta, \delta >0$ and $\nu \in (1,2)$ such that
\begin{equation}\label{06}
|W(t, u)| \geq \eta |u|^{\nu},\;\;\forall\; t\in \mathbb{I}\;\mbox{and}\;\;|u|\leq \delta.
\end{equation}
\end{enumerate} 

On the existence of solutions we have the following result.
\begin{Thm}\label{main1}
Assume that the conditions $(L_1)-(L_2)$ and $(W_1)-(W_3)$ hold. Then there exists $\Lambda >0$ such that for every $\lambda > \Lambda$, problem (\ref{01}) has at least one weak solution $u_\lambda$.
\end{Thm}

For technical reason, we consider that there exists $0<\mathbb{T}<+\infty$, such that $\mathbb{I} = (0,\mathbb{T})$, where $\mathbb{I}$ is given by $(L_3)$. On the concentration of solutions we have the following result.
\begin{Thm}\label{main2}
Let $u_\lambda$ be  a solution of problem (\ref{01}) obtained in Theorem \ref{01}, then $u_\lambda \to \tilde{u}$ strongly in $H^{\alpha}(\mathbb{R})$ as $\lambda \to \infty$, where $\tilde{u}$ is a nontrivial weak solution of the equation
\begin{eqnarray}\label{08}
&{_{t}}D_{\mathbb{T}}^{\alpha} {_{0}}D_{t}^{\alpha}u  = \nabla W(t, u),\quad t\in (0, \mathbb{T}),\\
& u(0) = u(\mathbb{T}) = 0.\nonumber
\end{eqnarray}
\end{Thm}

The rest of the paper is organized as follows: In section \S 2, we describe the Liouville-Weyl fractional calculus and we introduce the fractional space that we use in our work and some proposition are proven which will aid in our analysis. In section \S 3, we prove Theorem \ref{main1}. Finally in section \S 4 we prove Theorem \ref{main2}.


\section{Preliminary Results}

\subsection{Liouville-Weyl Fractional Calculus}

We first introduce some basic definitions of fractional calculus. The Liouville-Weyl fractional integrals of order $0<\alpha < 1$ are defined as
\begin{equation}\label{LWeq01}
_{-\infty}I_{x}^{\alpha}u(x) = \frac{1}{\Gamma (\alpha)} \int_{-\infty}^{x}(x-\xi)^{\alpha - 1}u(\xi)d\xi,
\end{equation}
and 
\begin{equation}\label{LWeq02}
_{x}I_{\infty}^{\alpha}u(x) = \frac{1}{\Gamma (\alpha)} \int_{x}^{\infty}(\xi - x)^{\alpha - 1}u(\xi)d\xi.
\end{equation}
The Liouville-Weyl fractional derivative of order $0<\alpha <1$ are defined as the left-inverse operators of the corresponding Liouville-Weyl fractional integrals
\begin{equation}\label{LWeq03}
_{-\infty}D_{x}^{\alpha}u(x) = \frac{d }{d x} {_{-\infty}}I_{x}^{1-\alpha}u(x),
\end{equation}
and
\begin{equation}\label{LWeq04}
_{x}D_{\infty}^{\alpha}u(x) = -\frac{d }{d x} {_{x}}I_{\infty}^{1-\alpha}u(x).
\end{equation}

\noindent
Recall that the Fourier transform $\widehat{u}(w)$ of $u(x)$ is defined by
$$
\widehat{u}(w) = \int_{-\infty}^{\infty} e^{-ix.w}u(x)dx.
$$
Let $u(x)$ be defined on $(-\infty, \infty)$. Then the Fourier transform of the Liouville-Weyl integral and differential operator satisfies
\begin{equation}\label{LWeq05}
\widehat{ _{-\infty}I_{x}^{\alpha}u(x)}(w) = (iw)^{-\alpha}\widehat{u}(w),\;\; \widehat{ _{x}I_{\infty}^{\alpha}u(x)}(w) = (-iw)^{-\alpha}\widehat{u}(w),
\end{equation}
\begin{equation}\label{LWeq06}
\widehat{ _{-\infty}D_{x}^{\alpha}u(x)}(w) = (iw)^{\alpha}\widehat{u}(w),\;\;\widehat{ _{x}D_{\infty}^{\alpha}u(x)}(w) = (-iw)^{\alpha}\widehat{u}(w).
\end{equation}
\subsection{Fractional Derivative Space}

Our aim is to establish a variational structure that enables us to reduce the existence of solutions of (\ref{01}) to finding critical points of the corresponding functional, and it is necessary to construct appropriate function spaces.

We first introduce some fractional-index spaces. Denote by $L^{p}(\mathbb{R})$, $p\in [2,+\infty]$, the Banach space of functions on $\mathbb{R}$ with values in $\mathbb{R}$ with the norm
$$
\|u\|_{L^p} = \left( \int_{\mathbb{R}} |u(t)|^pdt\right)^{1/p},
$$
and $L^{\infty}(\mathbb{R})$ is the Banach space of essentially bounded functions from $\mathbb{R}$ into $\mathbb{R}$ equipped with the norm
$$
\|u\|_{\infty} :=ess\sup\{|u(t)|:t\in \mathbb{R}\}.
$$

Let $0< \alpha \leq 1$ and $1<p<\infty$. The fractional derivative space $E_{0}^{\alpha ,p}$ is defined by the closure of $C_{0}^{\infty}([0,T], \mathbb{R}^n)$ with respect to the norm
$$
\|u\|_{\alpha ,p} = \left(\int_{0}^{T} |u(t)|^pdt + \int_{0}^{T}|{_{0}}D_{t}^{\alpha}u(t)|^pdt  \right)^{1/p}, \;\;\forall\; u\in E_{0}^{\alpha ,p}.
$$
This space can be characterized by $E_{0}^{\alpha , p} = \{u\in L^{p}([0,T], \mathbb{R}^n)/\;\; {_{0}}D_{t}^{\alpha}u \in L^{p}([0,T], \mathbb{R}^n)\;\mbox{and}\;u(0) = u(T) = 0\}$. Moreover $(E_{0}^{\alpha ,p}, \|.\|_{\alpha ,p})$ is a reflexive and separable Banach space.

\noindent
For $\alpha > 0$, define the semi-norm
$
|u|_{I_{-\infty}^{\alpha}} = \|_{-\infty}D_{x}^{\alpha}u\|_{L^{2}},
$
and norm
\begin{equation}\label{FDEeq01}
\|u\|_{I_{-\infty}^{\alpha}} = \left( \|u\|_{L^{2}}^{2} + |u|_{I_{-\infty}^{\alpha}}^{2} \right)^{1/2},
\end{equation}
and let
$$
I_{-\infty}^{\alpha} (\mathbb{R}) = \overline{C_{0}^{\infty}(\mathbb{R}, \mathbb{R}^n)}^{\|.\|_{I_{-\infty}^{\alpha}}},
$$
where $C_{0}^{\infty}(\mathbb{R}, \mathbb{R}^n)$ denotes the space of infinitely differentiable functions from $\mathbb{R}$ into $\mathbb{R}^n$ with vanishing property at infinity.

Now we can define the fractional Sobolev space $H^{\alpha}(\mathbb{R}, \mathbb{R}^n)$ in terms of the Fourier transform. Choose $0< \alpha < 1$, define the semi-norm
\begin{equation}\label{FDEeq02}
|u|_{\alpha} = \||w|^{\alpha}\widehat{u}\|_{L^{2}},
\end{equation}
and the norm
$$
\|u\|_{\alpha} = \left( \|u\|_{L^{2}}^{2} + |u|_{\alpha}^{2} \right)^{1/2},
$$
and let
$$
H^{\alpha}(\mathbb{R}) = \overline{C_{0}^{\infty}(\mathbb{R}, \mathbb{R}^n)}^{\|.\|_{\alpha}}.
$$
Moreover, we note a function $u\in L^{2}(\mathbb{R}, \mathbb{R}^n)$ belongs to $I_{-\infty}^{\alpha}(\mathbb{R}, \mathbb{R}^n)$ if and only if $|w|^{\alpha}\widehat{u} \in L^{2}.$
We have
\begin{equation}\label{FDEeq04}
|u|_{I_{-\infty}^{\alpha}} = \||w|^{\alpha}\widehat{u}\|_{L^{2}}.
\end{equation}
Therefore $I_{-\infty}^{\alpha}(\mathbb{R}, \mathbb{R}^n)$ and $H^{\alpha}(\mathbb{R}, \mathbb{R}^n)$ are equivalent with equivalent semi-norm and norm. 

Analogous to $I_{-\infty}^{\alpha}(\mathbb{R}, \mathbb{R}^n)$ we introduce $I_{\infty}^{\alpha}(\mathbb{R}, \mathbb{R}^n)$. Let the semi-norm $|u|_{I_{\infty}^{\alpha}} = \|_{x}D_{\infty}^{\alpha}u\|_{L^{2}},$
and the norm
\begin{equation}\label{FDEeq05}
\|u\|_{I_{\infty}^{\alpha}} = \left( \|u\|_{L^{2}}^{2} + |u|_{I_{\infty}^{\alpha}}^{2} \right)^{1/2},
\end{equation}
and let
$$
I_{\infty}^{\alpha}(\mathbb{R}) = \overline{C_{0}^{\infty}(\mathbb{R},\mathbb{R}^n)}^{\|.\|_{I_{\infty}^{\alpha}}}.
$$
Moreover $I_{-\infty}^{\alpha}(\mathbb{R},\mathbb{R}^n)$ and $I_{\infty}^{\alpha}(\mathbb{R}, \mathbb{R}^n)$ are equivalent, with equivalent semi-norm and norm.

Let $C(\mathbb{R},\mathbb{R}^n)$ denote the space of continuous functions from $\mathbb{R}$ into $\mathbb{R}^n$. Then we obtain the following Sobolev lemma.
\begin{Thm}\label{FDEtm01}
\cite{CT} If $\alpha > \frac{1}{2}$, then $H^{\alpha}(\mathbb{R}, \mathbb{R}^n) \subset C(\mathbb{R}, \mathbb{R}^n)$ and there is a positive constant $C_{\alpha}$ such that
\begin{equation}\label{FDEeq06}
\|u\|_{\infty} \leq C_{\alpha} \|u\|_{\alpha}.
\end{equation}
\end{Thm}

In what follows, we introduce the fractional space in which we will construct the variational framework of (\ref{01}). Let
$$
X^{\alpha} = \left\{ u\in H^{\alpha}(\mathbb{R}, \mathbb{R}^{n})|\;\;\int_{\mathbb{R}} \left[|_{-\infty}D_{t}^{\alpha}u(t)|^{2} + (L(t)u(t), u(t))\right] dt < \infty  \right\},
$$
then $X^{\alpha}$ is a reflexive and separable Hilbert space with the inner product
$$
\langle u,v \rangle_{X^{\alpha}} = \int_{\mathbb{R}} \left[ _{-\infty}D_{t}^{\alpha}u(t) \cdot{_{-\infty}}D_{t}^{\alpha}v(t)+ (L(t)u(t),v(t))\right]dt
$$
and the corresponding norm
$$
\|u\|_{X^{\alpha}}^{2} = \langle u,u \rangle_{X^{\alpha}}.
$$
For $\lambda>0$, we also need the following inner product  
$$
\langle u,v \rangle_{\lambda} = \int_{\mathbb{R}} \left[ _{-\infty}D_{t}^{\alpha}u(t) \cdot{_{-\infty}}D_{t}^{\alpha}v(t)+ \lambda (L(t)u(t), v(t))\right]dt,
$$
and the corresponding norm $\|u\|_{\lambda}^2 = \langle u,u \rangle_{\lambda}$. It is clear that $\|u\|_{X^\alpha} \leq \|u\|_{\lambda}$ for $\lambda \geq 1$. Set $X_{\lambda} = (X^{\alpha}, \|.\|_{\lambda})$.

\begin{Lem}\label{lmx1}
Suppose that $(L_1)-(L_2)$ hold. Then, the embedding $X^{\alpha} \hookrightarrow H^{\alpha}(\mathbb{R}, \mathbb{R}^n)$ is continuous.
\end{Lem}
\begin{proof}
By $(L_1)-(L_2)$ and (\ref{FDEeq06}), we have
\begin{eqnarray*}
&&\int_{\mathbb{R}}|u(t)|^2dt = \int_{\{l<k\}}|u(t)|^2dt + \int_{\{l\geq k\}} |u(t)|^2dt \\
&&\leq \|u\|_{\infty}^{2}|\{l<k\}| + \frac{1}{k}\int_{\mathbb{R}}l(t)|u(t)|^2dt\\
&&\leq C_{\alpha}^{2}|\{l<k\}|\left(\int_{\mathbb{R}}(|{_{-\infty}}D_{t}^{\alpha}u(t)|^2 + |u(t)|^2) dt  \right) + \frac{1}{k}\int_{\mathbb{R}}(L(t)u(t), u(t))dt. 
\end{eqnarray*}
Therefore
\begin{equation}\label{X02}
\|u\|_{L^2}^{2} \leq \frac{\max\{C_{\alpha}^2|\{l<k\}|, \frac{1}{k}\}}{1-C_{\alpha}^{2}|\{l<k\}|}\|u\|_{X^\alpha}^{2}.
\end{equation}
Then, by (\ref{X02}) we get
\begin{equation}\label{X03}
\|u\|_{\alpha}^{2} \leq \left( 1 + \frac{\max\{C_{\alpha}^{2}|\{l<k\}|, \frac{1}{k}\}}{1-C_{\alpha}^{2}|\{l<k\}|} \right)\|u\|_{X^\alpha}^{2},
\end{equation}
which implies that the embedding $X^{\alpha}  \hookrightarrow  H^{\alpha}(\mathbb{R})$ is continuous . 
\end{proof}

\begin{Lem}\label{lmx2}
Suppose that $(L_1)-(L_2)$ hold. Then, there exists $\Lambda >0$ such that, for all $\lambda \geq \Lambda$,  the embedding $X_{\lambda} \hookrightarrow L^{r}(\mathbb{R}, \mathbb{R}^n)$ is continuous for all $2\leq r < \infty$. 
\end{Lem}
\begin{proof}
Let $\Lambda = \frac{1}{kC_{\alpha}^{2}|\{L<k\}|}$. Using the same ideas of the proof of Lemma \ref{lmx1}, for 
all $\lambda \geq \Lambda$ we also obtain
\begin{equation}\label{X04}
\|u\|_{L^2}^{2} \leq \frac{C_{\alpha}^{2}|\{L<k\}|}{1-C_{\alpha}^{2}|\{L<k\}|}\|u\|_{\lambda}^{2} = \frac{1}{\theta_0}\|u\|_{\lambda}^{2},
\end{equation}
where $\theta_0 = \frac{1-C_{\alpha}^{2}|\{L<k\}|}{C_{\alpha}^{2}|\{L<k\}|}$. Furthermore, using (\ref{X04}), for each $r\in (2,\infty)$ and $\lambda \geq \Lambda$ we have
\begin{eqnarray*}
\int_{\mathbb{R}}|u(t)|^{r}dt &\leq& \|u\|_{\infty}^{r-2}\int_{\mathbb{R}}|u(t)|^2dt\\
&\leq&C_{\alpha}^{r-2}\left( \int_{\mathbb{R}} |{_{-\infty}}D_{t}^{\infty}u(t)|^2 + |u(t)|^2dt \right)^{\frac{r-2}{2}}\frac{C_{\alpha}^{2}|\{l<k\}|}{1-C_{\alpha}^{2}|\{l<k\}|}\|u\|_{\lambda}^{2}\\
&\leq& \frac{1}{\theta_{0}^{r/2}|\{l<k\}|^{\frac{r-2}{2}}} \|u\|_{\lambda}^{r}.
\end{eqnarray*} 
Therefore, for all $r\in (2,\infty)$
\begin{equation}\label{X05}
\|u\|_{L^r}^{r} \leq \frac{1}{\theta_{0}^{r/2}|\{L<k\}|^{\frac{r-2}{2}}}\|u\|_{\lambda}^{r}.
\end{equation}
\end{proof}

In order to prove Theorem \ref{main1}, we use the following result by Rabinowitz \cite{PR}

\begin{Lem}\label{Rlem}
Let $\mathfrak{E}$ be a real Banach space and $\Phi \in C^{1}(\mathfrak{E}, \mathbb{R})$ satisfy the (PS)-condition. If $\Phi$ is bounded from below, then $c = \inf_{\mathfrak{E}}\Phi$ is a critical value of $\Phi$.
\end{Lem}

\section{Proof of Theorem \ref{main1}}
It is well known that (\ref{01}) is variational and its solutions are the critical points of the functional $I_{\lambda}$ defined in $X_{\lambda}$ by
$$
I_{\lambda}(u) = \frac{1}{2}\|u\|_{\lambda}^{2} - \int_{\mathbb{R}}W(t,u)dt.
$$
Furthermore, it is easy to prove that the functional $I_{\lambda}$ is of class $C^1$ in $X_{\lambda}$, and that  
$$
I'_{\lambda}(u)\varphi = \int_{\mathbb{R}} [{_{-\infty}}D_{t}^{\alpha}u \cdot {_{-\infty}}D_{t}^{\alpha}\varphi + \lambda (L(t)u(t), \varphi (t)) ]dt - \int_{\mathbb{R}}(\nabla W(t,u), \varphi) dt
$$

First, we give some useful lemmas.
\begin{Lem}\label{elm1}
Assume that $(L_1)-(L_2)$, $(W_1)-(W_2)$ hold. Then, for all $\lambda \geq \Lambda$, $I_{\lambda}$ is bounded from below in $X_\lambda$
\end{Lem}

\begin{proof}
By ($W_1$), (\ref{X04}) and H\"older inequality, we get
\begin{eqnarray*}
I_{\lambda} (u) & = & \frac{1}{\lambda}\|u\|_{\lambda}^{2} - \int_{\mathbb{R}}W(t,u)dt\\
&\geq& \frac{1}{2}\|u\|_{\lambda}^{2} - \frac{1}{p}\int_{\mathbb{R}} \xi (t)|u(t)|^pdt\\
&\geq& \frac{1}{2}\|u\|_{\lambda}^{2} - \frac{1}{p} \|\xi\|_{L^{\frac{2}{2-p}}} \|u\|_{L^2}^{p}\\
&\geq& \frac{1}{2}\|u\|_{\lambda}^{2} - \frac{1}{p\theta_{0}^{p/2}}\|\xi\|_{L^{\frac{2}{2-p}}}\|u\|_{\lambda}^{p},
\end{eqnarray*}
which implies that $I_{\lambda}(u) \to +\infty$ as $\|u\|_{\lambda} \to +\infty$, since $1<p<2$. Consequently $I_{\lambda}$ is a functional bounded from below in $X_{\lambda}$.
\end{proof}

\begin{Lem}\label{elm2}
Suppose that $(L_1)-(L_2)$, $(W_1)$ and $(W_2)$ are satisfied. Then $I_{\lambda}$ satisfies the (PS)-condition for each $\lambda \geq \Lambda$.  
\end{Lem}

\begin{proof}
Assume that $\{u_n\} \in X_{\lambda}$ is a sequence such that $I_{\lambda}(u_n)$ is bounded and $IÕ_{\lambda}(u_n) \to 0$ as $n\to \infty$. By Lemma \ref{elm1}, it is clear that $\{u_n\}$ is bounded in $X_{\lambda}$. Thus, there exists a constant $\mathfrak{C}>0$ such that 
\begin{equation}\label{e01}
\|u_n\|_{L^r} \leq \frac{1}{\theta_{0}^{1/2}|\{L<k\}|^{\frac{r-2}{2r}}}\|u_n\|_{\lambda} \leq \mathfrak{C}, \;\;\mbox{for all}\;\lambda \geq \Lambda,
\end{equation}
where $r\in [2, \infty]$. Passing to a subsequence if necessary, we may assume that $u_n \rightharpoonup u$ weakly in $X_{\lambda}$. For any $\epsilon >0$, since $\xi(t) \in L^{\frac{2}{2-p}}(\mathbb{R})$, we can choose $T>0$ such that
\begin{equation}\label{e02}
\left(\int_{|t|>T} |\xi (t)|^{\frac{2}{2-p}}dt \right)^{\frac{2-p}{2}} < \epsilon. 
\end{equation}  
Moreover, since $u_n \to u\quad\mbox{in}\;\;L_{loc}^{\infty}(\mathbb{R}, \mathbb{R}^n),$ we get $u_n \to u\quad\mbox{in}\;\;L_{loc}^{2}(\mathbb{R}, \mathbb{R}^n)$. Hence
\begin{equation}\label{e03} 
\lim_{n \to \infty}\int_{|t|\leq T} |u_n(t) - u(t)|^2dt = 0.
\end{equation}
Therefore, from (\ref{e03}), there exists $n_0\in \mathbb{N}$ such that
\begin{equation}\label{e04}
\lim_{n\to \infty} \int_{|t|<T} |u_n(t) - u(t)|^2dt < \epsilon^2, \;\;\mbox{for}\;\;n \geq n_0.
\end{equation}
Hence, by ($W_1$), (\ref{e01}), (\ref{e04}) and the H\"older inequality, for any $n\geq n_0$, we have
\begin{eqnarray}\label{e05}
&&\int_{|t|\leq T} |\nabla W(t, u_n(t)) - \nabla W(t, u(t))||u_n(t) - u_n(t)|dt \nonumber\\
&&\leq \left(\int_{|t|\leq T} |\nabla W(t, u_n(t)) - \nabla W(t, u(t))|^2 \right)^{1/2}\left( \int_{|t|\leq T} |u_n(t) - u(t)|^2dt \right)^{1/2}\nonumber\\
&&\leq \epsilon \left( \int_{|t|\leq T} 2(|\nabla W(t, u_n(t))|^2 + |\nabla W(t, u(t))|^2)dt \right)^{1/2}\nonumber\\
&&\leq 2\epsilon \left( \int_{|t|\leq T}|\xi(t)|^2\left( |u_n(t)|^{2(p-1)} + |u(t)|^{2(p-1)} \right)dt \right)^{1/2}\nonumber \\
&&\leq 2\epsilon \left[\|\xi\|_{L^{\frac{2}{2-p}}} \left( \|u_n\|_{L^2}^{2(p-2)} + \|u\|_{L^2}^{2(p-1)} \right) \right]^{1/2}\nonumber\\
&&\leq 2\epsilon \left[ \|\xi\|_{L^{\frac{2}{2-p}}}^{2}\left(  \mathfrak{C}^{2(p-1)} + \|u\|_{L^2}^{2(p-1)} \right)\right]^{1/2}.
\end{eqnarray} 
On the other hand, by (\ref{e01}), (\ref{e02}), (\ref{e04}) and ($W_1$), we have
\begin{eqnarray}\label{e06} 
&&\int_{|t|>T} |\nabla W(t, u_n(t)) - \nabla W(t, u(t))||u_n(t) - u(t)|dt\nonumber\\
&& \leq 2\int_{|t|>T} |\xi (t)|(|u_n(t)|^p + |u(t)|^p)dt\nonumber\\
&&\leq 2\epsilon \frac{1}{\theta_0^{p/2}} (\|u_n\|_{\lambda}^{p} + \|u\|_{\lambda}^{p})\nonumber\\
&&\leq \frac{2\epsilon}{\theta_0^{p/2}}(\mathfrak{K}^p + \|u\|_{\lambda}^{p}).
\end{eqnarray}
Since $\epsilon$ is arbitrary, combining (\ref{e05}) and (\ref{e06}), we have
\begin{equation}\label{e07}
\int_{\mathbb{R}} |\nabla W(t, u_n(t)) - \nabla W(t, u(t))||u_n(t) - u(t)|dt < \epsilon,
\end{equation} 
as $n\to \infty$. Hence,
\begin{equation}\label{e08}
\langle I'_{\lambda}(u_n) - I'_{\lambda}(u), u_n - u\rangle =  \|u_n - u\|_{\lambda}^{2} + \int_{\mathbb{R}} (\nabla W(t, u_n(t)) - \nabla W(t, u(t)))(u_n(t) - u(t))dt.
\end{equation}
From, $\langle IÕ_{\lambda}(u_n) - IÕ_{\lambda}(u), u_n - u \rangle \to 0$, (\ref{e07}) and (\ref{e08}), we get $u_n \to u$ strongly in $X_{\lambda}$. Hence, $I_{\lambda}$ satisfies (PS)-condition.
\end{proof}

\noindent
{\bf Proof of Theorem \ref{main1}:} From Lemmas \ref{Rlem}, \ref{elm1}, \ref{elm2}, we know that $c_{\lambda} = \inf_{X_{\lambda}}I_{\lambda}(u)$ is a critical value of functional $I_{\lambda}$; that is, there exists a critical point $u_\lambda \in X_{\lambda}$ such that $I_{\lambda}(u_{\lambda}) = c_{\lambda}$. 

Finally, we show that $u_{\lambda} \neq 0$. Let $u_0 \in (W_0^{1,2}(\mathbb{I})\cap X_{\lambda} )\setminus \{0\}$ and $\|u_0\|_{\infty} \leq 1$, then by ($W_2$), we have
\begin{eqnarray}\label{e09}
I_{\lambda}(su_0) &=& \frac{1}{2}\|su_0\|_{\lambda}^{2} - \int_{\mathbb{R}} W(t, su_0(t))dt\nonumber\\
& \leq &\frac{s^2}{2}\|u_0\|_{\lambda}^{2} - \int_{\mathbb{I}}W(t, su_0(t))dt\nonumber\\
&\leq& \frac{s^2}{2}\|u_0\|_{\lambda}^{2} - \eta s^{\nu}\int_{\mathbb{I}}|u_0(t)|^{\nu}dt,\;\;0< s < \delta.
\end{eqnarray}
Since $1< \nu <2$, it follows from (\ref{e09}) that $I_{\lambda}(s u_0) <0$ for $s>0$ small enough. Hence $I_{\lambda}(u_{\lambda}) =  c_{\lambda} <0$, therefore, $u_{\lambda}$ is a nontrivial critical point of $I_{\lambda}$ and so $u_{\lambda}$ is a nontrivial weak solution of problem (\ref{01}). The proof is complete. $\Box$

\section{Concentration of Solutions}

In the following, we study the concentration of solution for problem (\ref{01}) as $\lambda \to \infty$. Define
$$
\tilde{c} = \inf_{w\in  E_{0}^{\alpha}} I_{\lambda}|_{E_{0}^{\alpha}}(w),
$$ 
where $I_{\lambda}|_{E_{0}^{\alpha}}$ is a restriction of $I_{\lambda}$ on $E_{0}^{\alpha}$; that is,
$$
I_{\lambda}|_{E_{0}^{\alpha}}(w) = \frac{1}{2}\int_{0}^{\mathbb{T}}|{_{0}}D_{t}^{\alpha}w(t)|dt - \int_{0}^{\mathbb{T}}W(t, w(t))dt,
$$
for $w\in H^{\alpha}(\mathbb{R})$. Similar to the proof of Theorem \ref{main1} it is easy to prove that $\tilde{c} <0$ can be achieved. Since $E_{0}^{\alpha} \subset X_{\lambda}$ for all $\lambda >0$, we get
$$
c_{\lambda} \leq \tilde{c}<0,\quad\mbox{for all}\;\;\lambda > \Lambda.
$$

\noindent
{\bf Proof of Theorem \ref{main2}:} We follow the arguments in \cite{TBAPZW}. For any sequence $\lambda_n \to \infty$, let $u_n = u_{\lambda_n}$ be the critical point of $I_{\lambda_n}$ obtained in Theorem \ref{main1}. Thus
\begin{equation}\label{c01}
I_{\lambda_n}(u_n) \leq \tilde{c}<0
\end{equation}
and 
\begin{eqnarray*}
I_{\lambda_n}(u_n) &=& \frac{1}{2}\|u_n\|_{\lambda_n}^{2} - \int_{\mathbb{R}} W(t, u_n(t))dt\\
&\geq& \frac{1}{2}\|u_n\|_{\lambda_n}^{2} - \frac{1}{p\theta_0^{p/2}}\|\xi\|_{L^{\frac{2}{2-p}}}\|u\|_{\lambda_n}^{p},
\end{eqnarray*}
which implies
\begin{equation}\label{c02}
\|u_n\|_{\lambda_n} \leq C,
\end{equation}
where the constant $C>0$ is independent of $\lambda_n$. Therefore, we may assume that $u_n \rightharpoonup \tilde{u}$ in $X_{\lambda}$ and $u_n \to \tilde{u}$ in $L_{loc}^{p}(\mathbb{R})$ for $2\leq p \leq \infty$. By Fatou's Lemma, we have
\begin{eqnarray*}
\int_{\mathbb{R}} l(t)|\tilde{u}(t)|^2dt &\leq& \liminf_{n\to \infty} \int_{\mathbb{R}} l(t)u_n^{2}(t)dt \\
&\leq& \liminf_{n \to \infty} \int_{\mathbb{R}} (L(t)u_n(t),u_n(t))dt\\
&\leq& \liminf_{n\to \infty} \frac{\|u_n\|_{\lambda_n}^{2}}{\lambda_n} = 0,
\end{eqnarray*}
thus $\tilde{u} = 0$ a.e. in $\mathbb{R} \setminus \mathbb{J}$, $\tilde{u} \in E_{0}^{\alpha}$ by ($L_2$). Now for any $\varphi \in C_{0}^{\infty}((0, \mathbb{T}), \mathbb{R}^n)$, since $\langle I'_{\lambda_n}(u_n), \varphi \rangle = 0$, it is easy to check that
$$
\int_{0}^{\mathbb{T}} {_{0}}D_{t}^{\alpha}\tilde{u}\cdot {_{0}}D_{t}^{\alpha}\varphi dt - \int_{0}^{\mathbb{T}} (\nabla W(t, \tilde{u}(t)), \varphi (t))dt = 0,
$$
that is, $\tilde{u}$ is a weak solution of (\ref{08}) by the density of $C_{0}^{\infty}((0, \mathbb{T}), \mathbb{R}^n)$ in $E_{0}^{\alpha}$.

Next, we show that $u_n(t) \to \tilde{u}(t)$ in $L^{p}(\mathbb{R})$ for $2\leq p < \infty$. Otherwise, by vanishing lemma (see Lemma 2.1 in \cite{CT1}) there exists $\delta >0$, $R_0>0$ and $t_n\in \mathbb{R}$ such that
$$
\int_{t_n - R_0}^{t_n + R_0} (u_n - \tilde{u})^2dt \geq \delta.
$$
Moreover, $t_n \to \infty$, hence $|(t_n - R_0, t_n + R_0) \cap \{l< k\}| \to 0$. By the H\"older inequality, we have
$$
\int_{(t_n - R_0, t_n + R_0)\cap \{l<k\}} |u_n - \tilde{u}|^2dt \leq |(t_n-R_0, t_n + R_0)\cap \{l<k\}|\|u_n-\tilde{u}\|_{\infty}^{2} \to 0.
$$ 
Consequently 
\begin{eqnarray*}
\|u_n\|_{\lambda_n}^{2} & \geq & \lambda_n k \int_{(t_n - R_0, t_n + R_0)\cap\{l\geq k\}}|u_n(t)|^2dt\\
& = &\lambda_n k\int_{(t_n-R_0, t_n+R_0)\cap \{l\geq k\}} |u_n(t) - \tilde{u}(t)|^2dt\\
& = &\lambda_nk\left( \int_{(t_n - R_0, t_n + R_0)} |u_n(t) - \tilde{u}(t)|^2dt - \int_{(t_n - R_0, t_n + R_0)\cap \{l<k\}} |u_n(t) - \tilde{u}(t)|^2dt  \right) + o(1)\\
&\to& \infty,
\end{eqnarray*}
which contradicts (\ref{c02}). By virtue of $\langle I'_{\lambda_n}(u_n), u_n \rangle = \langle  I'_{\lambda_n}(u_n), \tilde{u}\rangle = 0$ and the fact that $u_n(t) \to \tilde{u}(t)$ strongly in $L^p(\mathbb{R})$ for $2\leq p < \infty$, we have
$$
\lim_{n\to \infty} \|u_n\|_{\lambda_n}^{2} = \|\tilde{u}\|_{\lambda_n}^2.
$$
Hence, $u_n \to \tilde{u}$ strongly in $X_{\lambda}$. Moreover, from (\ref{c01}), we have $\tilde{u}\neq 0$. This completes the proof.

\medskip


\end{document}